\documentclass[12pt,a4paper]{article}
\usepackage[active]{srcltx}
\usepackage{amsfonts}
\usepackage{amsmath}
\usepackage{amsbsy}
\usepackage{amsxtra}
\usepackage{latexsym}
\usepackage{amssymb}
\usepackage{url}
\usepackage{cite}
\usepackage{pstricks,pst-node}
\makeatletter
\g@addto@macro\thesection.
\makeatother

\newtheorem{theorem}{Theorem}

\newtheorem{corollary}{Corollary}
\newtheorem{proposition}{Proposition}
\newtheorem{remark}{Remark}

\newtheorem{example}{Example}

\DeclareMathOperator{\cone}{cone}
\DeclareMathOperator{\intr}{int}

\DeclareMathOperator{\card}{card}
\DeclareMathOperator{\argmin}{argmin}
\DeclareMathOperator{\iso}{iso}

\newcommand{\lng}{\langle}
\newcommand{\rng}{\rangle}
\newcommand{\lf}{\left}
\newcommand{\rg}{\right}

\newcommand{\R}{\mathbb R}
\newcommand{\N}{\mathbb N}

\newcommand{\mc}{\mathcal}

\newcommand{\f}{\frac}

\newenvironment{proof}{{\noindent\bf Proof.}}{\hfill$\Box$\\}

\begin{document}

\title{Isotonic regression and isotonic projection
\thanks{{\it 2010 AMS Subject Classification:} Primary 62J99,
Secondary 90C20; {\it Keywords and phrases:} isotonic regression, isotonic projection, quadratic optimization, ordered Euclidean space}
}
\author{A. B. N\'emeth\\Faculty of Mathematics and Computer Science\\Babe\c s Bolyai University, Str. Kog\u alniceanu nr. 1-3\\RO-400084 Cluj-Napoca, Romania\\email: nemab@math.ubbcluj.ro \and S. Z. N\'emeth\\School of Mathematics, The University of Birmingham\\The Watson Building, Edgbaston\\Birmingham B15 2TT, United Kingdom\\email: nemeths@for.mat.bham.ac.uk}
\date{}
\maketitle

\begin{abstract}

The note describes the cones in the Euclidean space
admitting isotonic metric projection with respect to
the coordinate-wise ordering. As  a consequence it is showed
that the metric projection onto the isotonic regression cone
(the cone defined by the general isotonic regression problem) admits
a projection which is isotonic with respect to the coordinate-wise 
ordering. 

\end{abstract}

\section{Introduction}

The isotonic regression problem \cite{Kruskal1964,BarlowBartholomewBremnerBrunk1972,RobertsonWrightDykstra1988,BestChakravarti1990,WuWoodroofeMentz2001,DeLeeuwHornikMair2009,ShivelySagerWalker2009} 
and its solution is intimately related to the
metric projection into a cone of the Euclidean vector space. In fact the
isotonic regression problem is a special quadratic optimization problem. 
It is desirable to relate the metric projection onto a closed
convex set to some order theoretic properties of the projection itself, which
can facilitate the solution of some problems. When the underlying set is a convex
cone, then the most natural is to consider the order relation
defined by the cone itself. This approach gives rise to the notion of the
isotonic projection cone, which by definition is a cone with the metric
projection onto it isotonic with respect to the order relation endowed by the cone itself.
As we shall see, the two notions of isotonicity, the first related to
the regression problem and the second to the metric projection, are at the first
sight rather different. The fact that the two notions are in fact
intimately related (this relation constitute the subject of this note) is 
somewhat accidental and it derives from semantical reasons.  

The relation of the two notions is observed and
taken advantage in the paper \cite{GuyaderJegouNemethNemeth2014}.
There was exploited the fact that the totally ordered isotonic regression cone
is an isotonic projection cone too. 

The problem occurs as a particular case of the following  
more general question: \emph{How does a closed convex set in the Euclidean space
which admits a metric projection isotonic with respect to some vectorial ordering on the space look like?}

It turns out, that the problem is strongly related to some
lattice-like operations defined on the space, and in particular
to the Euclidean vector lattice theory. (\cite{NemethNemeth2013})
When the ordering is the coordinate-wise one, the problem goes
back in the literature to \cite{Topkis1976,Veinott1981,QueyranneTardella2006,Isac1996,NishimuraOk2012}. However, we shall
ignore these connections in order to simplify the exposition.
Thus, the present note, besides proving some new results, has the role to bring together
some previous results and to present them in a simple unified form.

\section{Preliminaries} \label{preliminaries}

Denote by $\R^m$ the $m$-dimensional Euclidean space endowed with the scalar 
product $\lng\cdot,\cdot\rng:\R^m\times\R^m\to\R,$ and the Euclidean norm $\|.\|$ and topology 
this scalar product defines.

Throughout this note we shall use some standard terms and results from convex geometry 
(see e.g. \cite{Rockafellar1970}). 

Let $K$ be a \emph{convex cone} in $\R^m$, i. e., a nonempty set with
(i) $K+K\subset K$ and (ii) $tK\subset K,\;\forall \;t\in \R_+ =[0,+\infty)$.
The convex cone $K$ is called \emph{pointed}, if $K\cap (-K)=\{0\}.$
The cone $K$ is {\it generating} if $K-K=\R^m$. $K$ is generating
if and only if $\intr K\not= \emptyset.$

A closed, pointed generating convex cone is called \emph{proper}.
 
For any $x,y\in \R^m$, by the equivalence $x\leq_K y\Leftrightarrow y-x\in K$, the 
convex cone $K$ 
induces an {\it order relation} $\leq_K$ in $\R^m$, that is, a binary relation, which is 
reflexive and transitive. This order relation is {\it translation invariant} 
in the sense that $x\leq_K y$ implies $x+z\leq_K y+z$ for all $z\in \R^m$, and 
{\it scale invariant} in the sense that $x\leq_Ky$ implies $tx\leq_K ty$ for any $t\in \R_+$.
Conversely, if $\preceq$ is a translation invariant and scale invariant order relation on $\R^m$, then 
$\preceq=\leq_K$ with $K=\{x\in\R^m:0\preceq x\}$ a convex cone. If $K$ is pointed, then $\leq_K$ is 
\emph{antisymmetric} too, that is $x\leq_K y$ and $y\leq_K x$ imply that $x=y.$ Conversely, if the 
translation invariant and scale invariant order relation $\preceq$ on $\R^m$ is also antisymmetric, then the convex cone 
$K=\{x\in\R^m:0\preceq x\}$ is also pointed. (In fact it would be more appropriate to call the reflexive and transitive binary relations preorder
relations and the reflexive transitive and antisymmetric binary relations partial order relations. However, for simplicity of the terminology we 
decided to call both of them order relations.)


The set
$$ K= \cone\{x_1,\dots,x_m\}:=\{t^1x_1+\dots+t^m x_m:\;t^i\in \R_+,\;i=1,\dots,m\}$$
with $x_1,\,\dots,\,x_m$ linearly independent vectors is called a \emph{simplicial cone}.
A simplicial cone is closed, pointed and generating.

The \emph{dual} of the convex cone $K$ is the set
$$K^*:=\{y\in \R^m:\;\lng x,y\rng \geq 0,\;\forall \;x\in K\},$$
with $\lng\cdot,\cdot\rng $ the standard scalar product in $\R^m$.

The cone $K$ is called \emph{self-dual}, if $K=K^*.$ If $K$
is self-dual, then it is a generating, pointed, closed convex cone.

In all that follows we shall suppose that $\R^m$ is endowed with a
Cartesian reference system with the standard unit vectors $e_1,\dots,e_m$.
That is, $e_1,\dots,e_m$ is an orthonormal system of vectors in the sense
that $\lng e_i,e_j\rng =\delta_i^j$, where $\delta_i^j$ is the Kronecker symbol.
Then, $e_1,\dots,e_m$ form a basis of the vector space $\R^m$. If $x\in \R^m$, then
$$x=x^1e_1+\dots+x^me_m$$
can be characterized by the ordered $m$-tuple of real numbers $x^1,\dots,x^m$, called
\emph{the coordinates of} $x$ with respect the given reference system, and we shall write
$x=(x^1,\dots,x^m).$ With this notation we have $e_i=(0,\dots,0,1,0,\dots,0),$
with $1$ in the $i$-th position and $0$ elsewhere. Let
$x,y\in \R^m$, $x=(x^1,\dots,x^m)$, $y=(y^1,\dots,y^m)$, where $x^i$, $y^i$ are the coordinates of
$x$ and $y$, respectively with respect to the reference system. Then, the scalar product of $x$
and $y$ is the sum
$\lng x,y\rng =\sum_{i=1}^m x^iy^i.$

The set
\[\R^m_+=\{x=(x^1,\dots,x^m)\in \R^m:\; x^i\geq 0,\;i=1,\dots,m\}\]
is called the \emph{nonnegative orthant} of the above introduced Cartesian
reference system. A direct verification shows that $\R^m_+$ is a
self-dual cone. The order relation $\le_{\R^m_+}$ induced by $\R^m_+$ is called \emph{coordinate-wise ordering}.

Besides the non-negative orthant, given a Cartesian reference system,
the important class of \emph{isotonic regression cones} should be mentioned.
Let $w^i>0$, $i=1,\dots,m$ be weights and $(V=\{1,\dots,m\}, E)$ be a directed graph of vertices $V$ and edges $E\subset V\times V$ and
without loops (a so called \emph{simple directed graph}). (If $(i,j)\in E$, then $i$ is called its \emph{tail},
$j$ is called its \emph{head}.) 
 Then we shall call the set
 $$K^w_E=\lf\{ x\in \R^m:\;\f{x^i}{\sqrt{w^i}}\leq \f{x^j}{\sqrt{w^j}},\;\forall(i,j)\in E\rg\}$$
 the \emph{isotonic regression cone defined by the relations $E$ and the weights $w^i$}.  

If $(V,E)$ is connected directed simple graph for which each vertex is the
tail respective a head of at most one edge, then $K^w_E$ is called \emph{weighted monotone cone}.
In this case $K^w_E$ can be written (after a possible permutation of the standard unit vectors) in the form 
$$K^w_E=\lf\{ x\in \R^m:\;\f{x^1}{\sqrt{w^1}}\leq \f{x^2}{\sqrt{w^2}}\leq \dots\leq \f{x^m}{\sqrt{w^m}}\rg\}.$$


A \emph{hyperplane} (through $b\in\R^m$) is a set of form
\begin{equation}\label{hyperplane}
H(a,b)=\{x\in \R^m:\;\lng a,x\rng =\lng a,b\rng, \;a\not= 0\}.
\end{equation}
The nonzero vector $a$ in the above formula is called \emph{the normal}
of the hyperplane. 

A hyperplane $H(a,b)$ determines two \emph{closed half-spaces} $H_-(a,b)$ and
$H_+(a,b)$  of $\R^m$, defined by

\[H_-(a,b)=\{x\in \R^m:\; \lng a,x\rng \leq \lng a,b\rng\},\]
and
\[H_+(a,b)=\{x\in \R^m:\; \lng a,x\rng \geq \lng a,b\rng\}.\]

The cone $K\subset \R^m$ is called \emph{polyhedral} if it can be  represented in the form
\begin{equation}\label{poly}
K=\cap_{k=1}^n H_-(a_k,0).
\end{equation}

If $\intr K\not=\emptyset$, and the representation (\ref{poly}) is irredundant, then $K\cap H(a_k,0)$ is an 
$m-1$-dimensional convex cone $(k=1,\dots,n)$ and is called a \emph{facet of $K$}.

The simplicial cone and the isotonic regression cones are polyhedral.






\section{Metric projection and isotonic projection sets}

Denote by $P_D$ 
the projection mapping onto a nonempty closed convex set $D\subset \R^m,$ 
that is the mapping which associate
to $x\in \R^m$ the unique nearest point of $x$ in $D$ (\cite{Zarantonello1971}):

\[ P_Dx\in D,\;\; \textrm{and}\;\; \|x-P_Dx\|= \inf \{\|x-y\|: \;y\in D\}. \]

Given an order relation $\preceq$ in $\R^m$, the closed convex set is
said an \emph{isotonic projection set} if from $x\preceq y,\;x,\,y\in \R^m$,
it follows $P_Dx\preceq P_Dy$.

If $\preceq =\leq_K$ for some cone $K$, then the isotonic projection set $D$ is
called \emph{$K$-isotonic}.

If the cone $K$ is $K$-isotonic then it is called
an \emph{isotonic projection cone}. 

For $K=\R^m_+$ we have $P_Kx=x^+$ where $x^+$ is the vector
formed with the non-negative coordinates of $x$ and $0$-s in place of
negative coordinates. Since $x\leq_K y$ implies
$x^+\leq_K y^+$, it follows that $\R^m_+$ is an isotonic
projection cone.

We have the following geometric characterization of a
closed, generating isotonic projection cones (Theorem 1 and Corollary 1 in \cite{GuyaderJegouNemethNemeth2014}):

\begin{theorem}\label{isochar}
The closed generating cone $K\subset \R^m$ is an isotonic projection
cone if and only if its dual $K^*$ is a simplicial cone in the subspace it spans
generated by vectors with mutually non-acute angles.
\end{theorem}


\section{The nonnegative orthant and its isotonic projection subcones}

If $\R^m_+$ is the nonnegative
	orthant of a Cartesian system, then we have the following theorem (Corollaries 1 and 3 in \cite{NemethNemeth2013}):
	
\begin{theorem}\label{orthCinv}
	Let $C$ be a closed convex set with nonempty interior of the coordinate-wise ordered
	Euclidean space $\R^m$. Then, the following assertions are equivalent:
	\begin{enumerate}
	\item [\emph{(i)}]The projection $P_C$ is $\R^m_+$-isotonic;
	\item [\emph{(ii)}]\begin{equation}\label{orthCinv1}
	C=\cap_{i\in\N} H_-(a_i,b_i),
	\end{equation}
	where each hyperplane $H(a_i,b_i)$ is tangent to $C$ and the normals $a_i$ 
	are nonzero vectors $a_i=(a_i^1,\dots,a_i^m)$ with the properties 
		$a_i^ka_i^l\leq 0$
		whenever $k\not= l,\;\;i\in \N.$

\end{enumerate}
\end{theorem}

\begin{example}\label{3dimsimpl}
Consider the space $\R^3$ endowed with a Cartesian reference system,
and suppose
$$K_1=H_-((-2,1,0),0)\cap H_-((1,-2,0),0)\cap H_-((0,0,-1),0),$$
and
$$K_2=H_-((-2,1,0),0)\cap H_-((1,-2,0),0)\cap H_-((0,1,-1),0).$$
Then $K_1$ and $K_2$ are simplicial cones in $\R^3_+$, $x=(1,1,2)\in \intr K_i,\;i=1,2.$
Since
$$K_1= \cone \{(-2,1,0),(1,-2,0),(0,0,-1)\}^\perp$$
and
$$K_2= \cone \{(-2,1,0),(1,-2,0),(0,1,-1)\}^\perp,$$
using the main result in \cite{IsacNemeth1986} we see that $K_1$
is itself an isotonic projection cone, while $K_2$ is not. 
Obviously, $K_1$ and $K_2$ are both $\R^3_+$-isotonic projection sets.
\end{example}

\begin{example}\label{3dim}
Let us consider the space $\R^3$ endowed with a Cartesian reference system.

Consider the vectors
\begin{gather*}
	a_1=(-2,1,0),\;a_2=(1,-2,0),\;a_3=(-2,0,1),\;a_4=(1,0,-2),\; a_5=(0,-2,1),\\
	a_6=(0,1,-2).
\end{gather*}
Then,
$$K=\cap_{i=1}^6 H_-(a_i,0)\subset \R^3_+$$
is by Theorem \ref{orthCinv} an $\R^3_+$-isotonic projection cone with six facets.

Indeed, $\lng a_1,x\rng \leq 0$ and $\lng a_2,x\rng \leq 0$ imply that $x^1\geq 0$ and
$x^2\geq 0.$ We can similarly show that $x\in K$ yields $x^3\geq 0.$ Thus,
$K\subset\R^3_+$. For $y=(1,1,1)$ we have $\lng a_i,y\rng <0$. Hence $y\in \intr K$. It follows that $K$ is a proper cone and 
the sets $H(a_i,0)\cap K$, $i=1,\dots,6$ are different facets of $K$.
\end{example}

Next we shall show that the cone in Example \ref{3dim} is in some sense extremal among the $\R^3_+$-isotonic subcones in $\R^3_+$. More precisely
we have

\begin{theorem}\label{orthisosubcone}

If $K$ is a generating
cone in  $\R^m$, then it is $\R^m_+$-isotonic,
if and only if it is a polyhedral cone of the form
\begin{equation}\label{k1}
K= \cap_{k<l} (H_-(a_{kl1},0)\cap H_-(a_{kl2},0)), \;\;k,\,l\in \{1,\dots,m\}
\end{equation}
where $a_{kli}$ are nonzero vectors with $a_{kli}^ka_{kli}^l\leq 0$ and $a_{kli}^j =0$ for $j\notin \{k,l\},\;i=1,2.$
Hence $K$ possesses at most $m(m-1)$ facets. There exists a cone $K$ of the above form with exactly $m(m-1)$ facets.

\end{theorem}

\begin{proof}

	The sufficiency is an immediate consequence of Theorem \ref{orthCinv}. Next we prove the necessity.
Assume that $K$ is an $\R^m_+$-isotonic generating cone. By using the same Theorem \ref{orthCinv}, we have that
\begin{equation}\label{k}
	K=\cap_{i\in \mc J} H_-(a_i,0),
	\end{equation}
	where $\mc J\subset \N$ is a set of indices and
	where each hyperplane $H(a_i,0)$ is tangent to $K$ and the normals $a_i$ 
	are nonzero vectors $a_i=(a_i^1,\dots,a_i^m)$ with the properties 
		$a_i^ka_i^l\leq 0$
		whenever $k\not= l,\;\;i\in \N.$

First of all we introduce the notation
$$\mc A_{kl}=\{i: a_i^j=0,\,j\notin \{k,l\}\}, \;\;k,\,l\in \{1,\dots,m\} ,\;k<l.$$
(In Example \ref{3dim} $\mc A_{12}=\{1,2\} ,\; \mc A_{13}=\{3,4\},\; \mc A_{23}=\{5,6\}.$)

We claim that 
\begin{equation}\label{akl}
\mc A_{kl} \not= \emptyset,\; k<l,\; \textrm{and} \;\cup_{k<l} \mc A_{kl} =\mc J.
\end{equation}
This follows from the structure of the normals $a_i.$ Indeed if
$a_i$ possesses two non-zero components, say $a_i^k$ and $a_i^l$, $k<l$, then
$i\in \mc A_{kl}$. If it has only one non-zero component, say $a_i^k$ with $k<m$,
then $i\in \mc A_{km}$, or only one nonzero component $a_i^m$  then $i\in \mc A_{km}$
for $k<m$. 

Let us see that
\begin{equation}\label{kket}
\cap_{i\in \mc A_{kl}}H_-(a_i,0)=H_-(a_{i_1},0)\cap H_-(a_{i_2},0),
\end{equation}
where $H_-(a_{i_j},0)$ are among those in (\ref{k})  and the case $i_1=i_2$ is possible.

Denote by $\R_{kl}$ the bidimensional subspace in $\R^m$ endowed by the $k$-th and $l$-th axis.
Then we have the representation
$$ \cap_{i\in \mc A_{kl}}H_-(a_i,0)= \R_{kl}^\perp \times (\cap_{i\in \mc A_{kl}}H_-(a_i,0))\cap \R_{kl}.$$ 
Now, $\cap_{i\in \mc A_{kl}}H_-(a_i,0))\cap \R_{kl}$ must be a two dimensional cone in $\R_{kl}$ (since $K$ is generating), hence
it must have one or  two extremal rays. That is the intersection can be expressed by one or two terms, that is,
we can suppose that $1\leq \card \mc A_{kl}\leq 2$ and (\ref{kket}) is proved.

With these remarks we can assert that the formula (\ref{k}) becomes
\begin{equation}\label{k1bis}
	K= \cap_{k<l}\lf(\cap_{i\in \mc A_{kl}}H_-(a_i,0)\rg)=\cap_{k<l} (H_-(a_{kl1},0)\cap H_-(a_{kl2},0)),
\end{equation}
where $a_{kli}^ka_{kli}^l\leq 0$ and $a_{kli}^j =0$ for $j\notin \{k,l\},\;\;i=1,2.$

From formula (\ref{k1bis}) it follows that in the representation (\ref{k}) of
$K$ there are at most $m(m-1)$ facets $H(a_i,0)\cap K$ of $K$. 

Using the construction in Example \ref{3dim} we can construct a $K$ with exactly
$m(m-1)$ facets. To this end, let for $k<l$
$a_{kl1}$ be the vector with $a_{kl1}^k=-2,\;\;a_{kl1}^l=1$ and $a_{kl1}^j=0$ for $j\notin \{k,l\},$
and $a_{kl2}$ be the vector with $a_{kl2}^k=1,\;\;a_{kl2}^l=-2$ and $a_{kl2}^j=0$ for $j\notin \{k,l\}.$
We have that the vectors $a_{kli}$ are pairwise non-parallel. Putting these vectors in the
representation (\ref{k1bis}) we get a proper subcone of $\R^m_+$ which is $\R^m_+$-isotonic
and possesses exactly $m(m-1)$ facets. Indeed, we must see  that in this case
the representation (\ref{k1bis}) is irredundant. But this follows from the fact that
$K\subset \R^m_+$ is a polyhedral cone with $x=(1,1,\dots,1)$ an interior point. Hence some of
$F_{kli}=H(a_{kli},0)\cap K$ must be facets of $K$. Now, from the special feature of $a_{kli}$ it
follows that the sets $F_{kli}$ are structurally equivalent and if one of them is
a facet, then all of them are so.

The proof also implies that $K$ must be a polyhedral cone and if its representation (\ref{k})
is irredundant, than the set $\mc J$ must be finite.

\end{proof}

\begin{remark}
 The
representation (\ref{k1bis}) can be redundant, even if the original one in (\ref{k}) is irredundant.
Indeed, $\R^m_+$ must be of form (\ref{k}) and its irredundant representation contains
$m$ terms, while its equivalent form (\ref{k1bis}) formally contains much more terms.
In this case (\ref{k1bis}) can contain $\f{m(m-1)}{2}$ terms. But even a minimal ``dual'' decomposition of
$\R^m_+$ is of cardinality $[\f{m+1}{2}]$ and hence it contains $2[\f{m+1}{2}]$ half-spaces. 

\end{remark}

\section{Every isotonic regression cone is an $\R^m_+$-isotonic projection set}

Projecting $y\in \R^m$ into $K$ given by (\ref{k1bis}) we have to solve the following
quadratic minimization problem:

\begin{equation}\label{projprob}
P_Ky= \argmin\lf\{\sum_{i=1}^m (x^i-y^i)^2:a_{kl1}^kx^k+a_{kl1}^lx^l\leq 0,\textrm{ }a_{kl2}^kx^k+a_{kl2}^lx^l\leq 0,\textrm{ }k<l\rg\},
\end{equation}

where the cases $a_{klj}^k=0$, or  $a_{klj}^l=0$ are not excluded.

By using Theorem \ref{orthCinv}, we see that, from
$$u\leq_{\R^m_+} v,$$
it follows that
$$P_Ku\leq_{\R^m_+} P_Kv.$$

A particular case of this projection problem
is equivalent to the so called \emph{isotonic regression problem}
\cite{Kruskal1964,BarlowBartholomewBremnerBrunk1972,RobertsonWrightDykstra1988,BestChakravarti1990,WuWoodroofeMentz2001,DeLeeuwHornikMair2009,ShivelySagerWalker2009}
which can be described as follows:

For a given $y\in \R^m$ and weights $w_i>0$, $i=1,\dots,m$

\[\iso(y):=\argmin\lf\{\sum_{i=1}^m w_i(x^i-y^i)^2:x^i\leq x^j,\textrm{ }\forall(i,j)\in E\rg\},\]
where $(V=\{1,\dots,m\}, E)$ is a directed simple graph.

Indeed,

\begin{align*}
	\iso(y)=\argmin\lf\{\sum_{i=1}^m \lf(\sqrt{w^i}x^i-\sqrt{w^i}y^i\rg)^2:\f{\sqrt{w^i}x^i}{\sqrt{w^i}}\leq \f{\sqrt{w^i}x^j}{\sqrt{w^j}},
	\textrm{ }\forall(i,j)\in E\rg\}\\=\f1{\sqrt{w}}P_{K^w_E}(\sqrt{w}y),
\end{align*}

where for any $z\in\R^m$ we denote \[\sqrt{w}z=(\sqrt{w^1}z^1,\dots,\sqrt{w^m}z^m)\] and 
\[\f z{\sqrt{w}}=\lf(\f{z^1}{\sqrt{w^1}},\dots,\f{z^m}{\sqrt{w^m}}\rg),\]
and $K^w_E$ is the isotonic regression cone defined in Section 2.

To compare with this with the general projection problem (\ref{projprob}), we observe that the restrictions on $x$ for $P_{K^w_E}(y)$ are of 
the form $$ a^i_{ij}x^i +a_{ij}^jx^j\leq 0$$
with $a_{ij}^i=1/\sqrt{w^i}$ and $a_{ij}^j =-1/\sqrt{w^j},\; (i,j)\in E.$
Thus we have established the

\begin{corollary}\label{regalt}
Every isotonic regression cone $K^w_E$
is an
$\R^m_+$-isotonic projection set.
\end{corollary}
We further have that

\begin{proposition}
The isotonic regression cone $K^w_E$ is an isotonic projection cone if and only if
in the oriented graph $(V,E)$ does not exist different edges with same tail or different edges
with same head, that is, edges of form
$(i,j)$ and $(i,k)$ with $j\not=k$, or edges of form $(i,j)$ and $(k,j)$ with $i\not= k.$
\end{proposition}

\begin{proof}
Assume e. g. that $(1,2),\;(1,3)\in E$. Then the corresponding normals are
$$ a_{1,2}=(1/\sqrt{w^1},-1/\sqrt{w^2},0,\dots,0)$$
and

$$ a_{1,3}=(1/\sqrt{w^1},0,-1/\sqrt{w^3},0,\dots,0).$$
 
Then $a_{1,2}$ and $a_{1,3}$ are normals in the irreducible representation of $K^w_E$,
and $\lng a_{1,2},a_{1,3}\rng >0.$ Thus, according to Theorem \ref{isochar}
$K^w_E$ cannot be an isotonic projection cone. Conversely, if there are no vertices
with the above type multiplicity property, then the normals in the irreducible
representation of $K^w_E$ (which in fact generates $-{K^w_{E}}^{*}$) form pair-wise
non-acute angles, hence by the same result 
$K^w_E$ is an isotonic projection cone.

\end{proof}

\begin{corollary}\label{moncon}
If $K^w_E$ is an isotonic projection cone, then $(V,E)$ splits in disjoint
union of connected simple graphs with vertices being
the tails or heads of at most one edge.
The single (up to a permutation of the canonical basis) isotonic regression cone $K^w_E$, with $(V,E)$
a directed connected simple graph, which is also an isotonic projection cone
is the weighted monotone cone.
\end{corollary}

\bibliographystyle{abbrv}
\bibliography{isoiso}
\end{document}